\UseRawInputEncoding
\documentclass[a4paper,12pt,reqno]{amsart}

\usepackage{amsmath}
\usepackage{amssymb}
\usepackage{amsfonts}
\usepackage{graphicx}
\usepackage[colorlinks]{hyperref}
\renewcommand\eqref[1]{(\ref{#1})} 

\graphicspath{ {images/} }
\title[Hardy and Rellich inequalities with Bessel Pairs]{Hardy and Rellich inequalities with Bessel Pairs}

\author[Michael Ruzhansky]{Michael Ruzhansky}
\address{\href{www.ruzhansky.org}{Michael Ruzhansky:}
	\endgraf
	Department of Mathematics: Analysis, Logic and Discrete Mathematics
	\endgraf
	Ghent University, Belgium
	\endgraf
	and
	\endgraf
	School of Mathematical Sciences
	\endgraf Queen Mary University of London 
	\endgraf
	United Kingdom
	\endgraf
	{\it E-mail address} {\rm Michael.Ruzhansky@ugent.be}
}

\author[Bolys Sabitbek]{Bolys Sabitbek}
\address{ \href{http://analysis-pde.org/bolys-sabitbek/}{Bolys Sabitbek:}
	\endgraf
		School of Mathematical Sciences
	\endgraf Queen Mary University of London 
	\endgraf
	United Kingdom
	\endgraf 
		and
	\endgraf 
	Al-Farabi Kazakh National University 
	\endgraf 
	71 al-Farabi Ave., Almaty, 050040 
		\endgraf
		Kazakhstan 
	\endgraf
	{\it E-mail address} {\rm b.sabitbek@qmul.ac.uk}
}


\thanks{ The authors were supported by EPSRC grant EP/R003025/2. The first	author was also supported by FWO Odysseus 1 grant G.0H94.18N: Analysis and Partial Differential Equations. The second author was aslo supported by the MES RK grant AP08053051.}

\newtheoremstyle{theorem}
{10pt}          
{10pt}  
{\sl}  
{\parindent}     
{\bf}  
{. }    
{ }    
{}     
\theoremstyle{theorem}

\numberwithin{equation}{section}
\theoremstyle{plain}
\newtheorem{thm}{Theorem}[section]

\newtheorem{cor}[thm]{Corollary}
\newtheorem{lem}[thm]{Lemma}

\theoremstyle{definition}

\newtheorem{rem}[thm]{Remark}

\newtheoremstyle{defi}
{10pt}          
{10pt}  
{\rm}  
{\parindent}     
{\bf}  
{. }    
{ }    
{}     
\theoremstyle{defi}

\begin{document}
		\begin{abstract}
	In this paper, we establish suitable characterisations for a pair of functions $(W(x),H(x))$ on a bounded, connected domain $\Omega \subset \mathbb{R}^n$ in order to have the following Hardy inequality
	\begin{equation*}
	\int_{\Omega} W(x) |\nabla u|_A^2 dx \geq \int_{\Omega} |\nabla d|^2_AH(x)|u|^2 dx, \,\,\, u \in C^{1}_0(\Omega),
	\end{equation*} 
	where $d(x)$ is a suitable quasi-norm (gauge), $|\xi|^2_A = \langle A(x)\xi, \xi \rangle$ for $\xi \in \mathbb{R}^n$ and $A(x)$ is an $n\times n$ symmetric, uniformly positive definite matrix defined on a bounded domain $\Omega \subset \mathbb{R}^n$. We also give its $L^p$ analogue. As a consequence, we present examples for a standard Laplacian on $\mathbb{R}^n$, Baouendi-Grushin operator, and sub-Laplacians on the Heisenberg group, the Engel group and the Cartan group. Those kind of characterisations for a pair of functions $(W(x),H(x))$ are obtained also for the Rellich inequality. These results answer the open problems of Ghoussoub-Moradifam \cite{GM_book}.
	\end{abstract}
	\maketitle
\section{Introduction}	
In the work \cite{GM11}, Ghoussoub and Moradifam gave necessary and sufficient conditions for a Bessel pair of positive radial functions $W(x)$ and $H(x)$ on a ball $B$ of radius $R$ in $\mathbb{R}^n$, $n\geq 1$, so that one has the Hardy inequality for all functions $u \in C^{\infty}_0(B)$
\begin{equation*}
	\int_{B} W(x) |\nabla u|^2 dx \geq \int_{B} H(x) |u|^2 dx,
\end{equation*}
and the Hardy-Rellich inequality for all functions $u \in C^{\infty}_0(B)$:
\begin{equation*}
	\int_{B} W(x)|\Delta u|^2 dx \geq \int_{B} H(x)|\nabla u|^2 dx +(n-1)\int_{B} \left( \frac{W(x)}{|x|^2}-\frac{W_r(|x|)}{|x|} \right)|\nabla u|^2 dx. 
\end{equation*} 
The characterisation of pairs of functions $W(x)$ and $H(x)$ made a very interesting connection between Hardy type inequalities and the oscillatory behavior of ordinary differential equations. Choosing suitable Bessel pairs $(W(x),H(x))$ allows one to improve, extend, and unify many results about Hardy and Hardy-Rellich inequalities that were established by Caffarelli et al. \cite{CKN84}, Brezis and Vazquez \cite{BV97}, Wang and Willem \cite{WW03}, Adimurthi et al. \cite{ACR02} and other authors.   
  
In the book \cite{GM_book}, Ghoussoub and Moradifam posed two questions:
\begin{itemize} 
	\item  Develop suitable characterisations for a pair of functions $(W(x),H(x))$ in order to have the following inequality
	\begin{equation*}
	\int_{\Omega} W(x) |\nabla u|_A^2 dx \geq \int_{\Omega} H(x)|u|^2 dx, \,\,\, u \in C^{1}_0(\Omega),
	\end{equation*} 
	where $|\xi|^2_A = \langle A(x)\xi, \xi \rangle$ for $\xi \in \mathbb{R}^n$ and $A(x)$ is an $n\times n$ symmetric, uniformly positive definite matrix defined on a bounded domain $\Omega \subset \mathbb{R}^n$. 
	\item  Devise a necessary and sufficient condition on a Bessel pair $(W(x),H(x))$ in order for the Rellich inequality to hold: 
		\begin{equation*}
	\int_{\Omega} W(x) |\Delta u|^2 dx \geq \int_{\Omega} H(x)|u|^2 dx, \,\,\, u \in C^{\infty}_0(\Omega).
	\end{equation*}
\end{itemize}

The aim of this paper is to give suitable characterisations for a Bessel pair of positive radial functions $W(x)$ and $H(x)$ for Hardy and Rellich inequalities on a bounded, connected domain $\Omega \subset \mathbb{R}^n$ that answers the open problems of Ghoussoub-Moradifam \cite{GM_book}. We prove Hardy and Rellich inequalities expressing conditions for Bessel pairs in terms of ordinary differential equations associated with the positive weight functions $W(x)$ and $H(x)$. Our approach relies on the first and second order Picone identities. This suggested approach seems very effective, allowing us to recover almost all well-known Hardy and Rellich type inequalities. This approach is an extension of the method of Allegretto-Huang \cite[Theorem 2.1]{AH98}, by adding the positive weight function $W(x)$. A similar approach was used by the authors \cite{RSS_NoDEA} to establish Hardy and Rellich type inequalities for general (real-valued) vector fields with boundary terms. Recently, in \cite{Cazacu20} Cazacu called this method (but without the function $W(x)$) as the {\it Method of Super-solutions in Hardy and Rellich inequalities} that was adopted from Davies \cite{Davies99}.

This characterisation of Bessel pairs builds an interesting bridge between Hardy (Rellich) type inequalities and ordinary differential equations. In particular, we can extend and improve many results for Hardy and Rellich type inequalities. Let us briefly recall several types of Hardy inequalities that can be recovered:
\begin{itemize}
	\item[I.] The classical Hardy inequality for $n\geq 3$ on a bounded domain $\Omega \subset \mathbb{R}^n$ asserts that 
	\begin{equation*}
		\int_{\Omega} |\nabla u|^2 dx \geq \left(\frac{n-2}{2}\right)^2 \int_{\Omega} \frac{|u|^2}{|x|^2} dx,\,\,\, u \in C^1_0(\Omega),
	\end{equation*}
	where the constant is optimal and not attained. This version of Hardy inequality was investigated by many authors \cite{KO90}, \cite{Davies99}, \cite{RS_book} and the references therein.
	\item[II.] The geometric Hardy inequality for any bounded convex domain $\Omega \subset \mathbb{R}^n$ with smooth boundary asserts that
	\begin{equation*}
			\int_{\Omega} |\nabla u|^2 dx \geq \frac{1}{4} \int_{\Omega} \frac{|u|^2}{\delta^2(x)} dx,\,\,\, u \in C^1_0(\Omega),
	\end{equation*}
	where $\delta(x):= dist(x,\partial \Omega)$ is the Euclidean distance to boundary $\partial \Omega$ and the constant is also optimal and not attained. There is a number of studies related to this subject, see e.g. \cite{Ancona}, \cite{Avka_Lap}, \cite{Avk_Wirth}, \cite{Davies99}, \cite{KO90} and \cite{Mazya85}.
	\item[III.] The multipolar Hardy inequality on a bounded domain $\Omega \subset \mathbb{R}^n$ asserts that 
	\begin{equation*}
		\int_{\Omega} |\nabla u|^2 dx \geq C \sum_{i=1}^k\int_{\Omega} \frac{|u|^2}{|x-a_i|^2} dx,\,\,\, u \in C^1_0(\Omega),
	\end{equation*}
	where $k$ is the number of poles. This type of inequalities was studied by Felli-Terracini \cite{FT06}, Bosi-Dolbeault-Esteban \cite{BDE08} and Cacazu-Zuazua \cite{CZ13}.
\end{itemize}

\section{Hardy inequalities with Bessel pairs}\label{sec_Hardy}
Let $\Omega \subset \mathbb{R}^n$ be a bounded domain with smooth boundary. Define 
\begin{equation}\label{eq-L_Ap}
\mathcal{L}_{p,A} f = - \sum_{i,j=1}^n \frac{\partial}{\partial {x_j}} \left( a_{ij}(x)|\nabla f|^{p-2}_A \frac{\partial f}{\partial x_j}  \right),
\end{equation}
and 
\begin{equation*}
|\nabla f|^2_A = \sum_{i,j=1}^n a_{ij}(x) \frac{\partial f}{\partial x_i} \frac{\partial f}{\partial x_j},
\end{equation*}
where $A(x)$ is an $n\times n$ symmetric, uniformly positive definite matrix with smooth coefficients defined on $\Omega$. 

Let $\Phi_p$ be a constant multiple of the fundamental solution (e.g. \cite{BG89} and \cite{KM92}) for $\mathcal{L}_{p,A}$ that solves the equation 
\begin{align*}
\mathcal{L}_{p,A} \Phi_p(x) &= 0, \,\,\, x \neq 0.
\end{align*}
From $\Phi_p$, we are able to define the quasi-norm
\begin{equation}\label{quasi-norm}
d(x) := \left\{\begin{matrix}
\Phi_p(x)^{\frac{p-1}{p-Q}}, & \text{for} \,\,\, x \neq 0,\\ 
0,& \text{for} \,\,\, x = 0,
\end{matrix}\right.
\end{equation}
where $Q$ is the appropriate homogeneous dimension and $1<p<Q$. 

Define 
\begin{equation}
\Psi_{\mathcal{L}_A}(x) := |\nabla d|^2_A(x),
\end{equation}
for $x \neq 0$. The function $\Psi_{\mathcal{L}_A}(x)$ can be calculated for the explicit form of the quasi-norm $d(x)$. For example:  
\begin{itemize}
	\item In the Euclidean setting, when $\mathcal{L}_A= \Delta$ is the standard Laplace operator, then $\Psi_{\Delta}(x) =1$.
	\item In the Heisenberg group, when $\mathcal{L}_A= \mathcal{L}_{\mathbb{H}}$ is the sub-Laplacian and the quasi-norm ($\mathcal{L}$-gauge) $d_{\mathbb{H}}(x)$, then $\Psi_{\mathcal{L}_{\mathbb{H}}}(x) = |x'|^2d_{\mathbb{H}}^{-2}$.
	\item For Baouendi-Grushin operator, when $\mathcal{L}_A= \mathcal{L}_{\gamma}$ is the Baouendi-Grushin operator and $d_{\gamma}(x)$ is associated the quasi-norm, then $\Psi_{\mathcal{L}_{\gamma}}(x) =|\xi|^{2\gamma}d_{\gamma}^{-2\gamma}$ where $x = (\xi,\zeta) \in \mathbb{R}^k\times \mathbb{R}^l$ and $\gamma>0$.
\end{itemize}
In the stratified Lie groups, we shall remark that the function $\Psi_{\mathcal{L}_A}(x)$ is $\delta_{\lambda}$-homogeneous degree of zero and translation invariant (i.e. $\Psi_{\mathcal{L}}(\alpha \circ x, \alpha \circ y)= \Psi_{\mathcal{L}}(x,y)$ for $x,y \in \mathbb{G}$ with $x\neq y$). Furthermore, the function $\Psi_{\mathcal{L}_A}(x)$ is the kernel of mean volume formulas (see more \cite[Definition 5.5.1]{BLU07}).  

Let us convert the equation
\begin{equation}\label{eq-PDE}
\sum_{i,j=1}^n \frac{\partial }{\partial x_j} \left( W(|x|) |\nabla v|^{p-2}_A a_{ij}(x) \frac{\partial v}{\partial x_i} \right) + |\nabla d|^p_AH(|x|) v^{p-1} =0,
\end{equation}
into the quasilinear second-order differential equation of the form
\begin{equation}\label{eq-ODE}
( r^{Q-1} W(r) (v'(r))^{p-1})' + r^{Q-1}H(r) v^{p-1}(r) =0,
\end{equation}
where $' = \partial_{r}$ and $r:=d(x)$.

Suppose that $W(x)$, $H(x)$, and $v(x)$ are positive radially symmetric functions. 
Let us rewrite the first term of \eqref{eq-PDE} in terms of the radial derivative. First note that for $i,j=1,\ldots,n$, we have
\begin{align}
\frac{\partial r}{\partial x_i} &= \left(\frac{p-1}{p-Q}\right) \Phi_p^{\frac{p-1}{p-Q}-1} \frac{\partial \Phi_p}{\partial x_i},\label{2.5}\\
\frac{\partial r}{\partial x_j}	\frac{\partial r}{\partial x_i} &= \left(\frac{p-1}{p-Q}\right)^2 \Phi_p^{2\frac{p-1}{p-Q}-2} \frac{\partial \Phi_p}{\partial x_i}\frac{\partial \Phi_p}{\partial x_j},\\
\frac{\partial^2 r}{\partial x_i \partial x_j} &= \left(\frac{p-1}{p-Q}\right) \Phi_p^{\frac{p-1}{p-Q}-1} \frac{\partial^2 \Phi_p}{\partial x_i\partial x_j} + \frac{(p-1)(Q-1)}{(p-Q)^2} \Phi_p^{\frac{p-1}{p-Q}-2} \frac{\partial \Phi_p}{\partial x_i}\frac{\partial \Phi_p}{\partial x_j}\\
&=  \left(\frac{p-1}{p-Q}\right) \Phi_p^{\frac{p-1}{p-Q}-1} \frac{\partial^2 \Phi_p}{\partial x_i\partial x_j} + \left(\frac{Q-1}{p-1}\right)\Phi_p^{-\frac{p-1}{p-Q}}\frac{\partial r}{\partial x_j}	\frac{\partial r}{\partial x_i}. \nonumber
\end{align} 
Then
\begin{align*}
\frac{\partial v}{\partial x_i} = v' \frac{\partial r}{\partial x_i}, \,\,\, \text{and} \,\,\, \frac{\partial^2 v}{\partial x_i \partial x_j} = \frac{\partial r}{\partial x_i}\frac{\partial r}{\partial x_j} v'' + \frac{\partial^2 r }{\partial x_i \partial x_j} v'.
\end{align*}
Since $\Phi_p = r^{\frac{p-Q}{p-1}}$, we thus have
\begin{align}
|\nabla r|^{p-2}_A &=\left( \frac{p-1}{p-Q}\right)^{p-2} |\nabla \Phi_p|_A^{p-2} r^{\frac{(Q-1)(p-2)}{p-1}},\\ 
|\nabla v|^{p-2}_A &=|\nabla r|^{p-2}_A ( v')^{p-2}, \\
\frac{\partial |\nabla v|^{p-2}_A}{\partial x_j}& = \frac{(Q-1)(p-2)}{(p-1)r}|\nabla r|^{p-2}_A (v')^{p-2} \frac{\partial r}{\partial x_j} + (p-2)|\nabla r|^{p-2}_A (v')^{p-3} v'' \frac{\partial r}{\partial x_j} \\
& +\left( \frac{p-1}{p-Q}\right)^{p-2} \frac{\partial |\nabla \Phi_p|_A^{p-2}}{\partial x_j}r^{\frac{(Q-1)(p-2)}{p-1}}(v')^{p-2}.\nonumber
\end{align}
Using above expressions, a straightforward computation gives 
\begin{align*}
\small
&\frac{\partial}{\partial {x_j}} \left( W a_{ij}(x)|\nabla v|^{p-2}_A \frac{\partial v}{\partial x_i}  \right) =   W a_{ij}(x)|\nabla v|^{p-2}_A \frac{\partial^2 v}{\partial x_i\partial x_j} \\
&  + a_{ij}(x) |\nabla v|^{p-2}_A \frac{\partial W}{\partial x_j}\frac{\partial v}{\partial x_i} + W  |\nabla v|^{p-2}_A \frac{\partial a_{ij}(x)}{\partial x_j}\frac{\partial v}{\partial x_i} + a_{ij}(x)W\frac{\partial |\nabla v|^{p-2}_A }{\partial x_j} \frac{\partial v}{\partial x_i} 
\end{align*}
\begin{align*}
	&=   W |\nabla r|^{p-2}_A ( v')^{p-2} v''\underbrace{a_{ij}(x) \frac{\partial r}{\partial x_i}\frac{\partial r}{\partial x_j}}_{=|\nabla r|^2_A} +  \left(\frac{p-1}{p-Q}\right) W |\nabla r|^{p-2}_A ( v')^{p-1} \Phi_p^{\frac{p-1}{p-Q}-1}a_{ij}(x) \frac{\partial^2 \Phi_p}{\partial x_i\partial x_j} \\
	&+ \left(\frac{Q-1}{p-1}\right) W |\nabla r|^{p-2}_A ( v')^{p-2}\Phi_p^{-\frac{p-1}{p-Q}}\underbrace{a_{ij}(x) \frac{\partial r}{\partial x_i}\frac{\partial r}{\partial x_j}}_{=|\nabla r|^2_A} +|\nabla r|^{p-2}_A ( v')^{p-1}W_r \underbrace{a_{ij}(x) \frac{\partial r}{\partial x_i}\frac{\partial r}{\partial x_j}}_{=|\nabla r|^2_A} \\ 
	&  + W |\nabla r|^{p-2}_A ( v')^{p-2} \frac{\partial a_{ij}(x)}{\partial x_j} v'\frac{\partial r}{\partial x_i} + \left( \frac{p-1}{p-Q}\right)^{p-2} W r^{\frac{(Q-1)(p-2)}{p-1}}(v')^{p-1}a_{ij}(x)  \frac{\partial |\nabla \Phi_p|_A^{p-2}}{\partial x_j}\frac{\partial r}{\partial x_i}\\
	& + \frac{(Q-1)(p-2)}{(p-1)r}|\nabla r|^{p-2}_A Ц(v')^{p-1} \underbrace{a_{ij}(x) \frac{\partial r}{\partial x_i}\frac{\partial r}{\partial x_j}}_{=|\nabla r|^2_A} + (p-2)W|\nabla r|^{p-2}_A (v')^{p-2} v'' \underbrace{a_{ij}(x) \frac{\partial r}{\partial x_i}\frac{\partial r}{\partial x_j}}_{=|\nabla r|^2_A}
\end{align*}
where $|\nabla r|^2_A = \sum_{i,j=1}^n \partial_{x_i}r \partial_{x_j}r$. We assume there is the summation $\sum_{i,j=1}^n$, to get
\begin{align*}
	&=  W |\nabla r|^{p}_A ( v')^{p-2} \left((p-1) v'' + \left(\frac{Q-1}{r}\right)\frac{1}{p-1} v' + \left(\frac{p-1}{p-Q}\right)\frac{\Phi_p^{\frac{p-1}{p-Q}-1}}{|\nabla r|_A^2} a_{ij}(x) \frac{\partial^2 \Phi_p}{\partial x_i \partial x_j} v'\right)\\
	& + |\nabla r|_A^p (v')^{p-1}W_r + W|\nabla r|_A^{p-2}(v')^{p-1} \frac{\partial a_{ij}}{\partial x_j}\frac{\partial r}{\partial x_i} + \left(\frac{Q-1}{r}\right)\left(1 - \frac{1}{p-1}\right)W |\nabla r|_A^p (v')^{p-1}\\
	&  + \underbrace{\left( \frac{p-1}{p-Q}\right)^{p-2}  r^{\frac{(Q-1)(p-2)}{p-1}} |\nabla \Phi_p|^{p-2}_A}_{=|\nabla r|_A^{p-2}} \frac{W(v')^{p-1}a_{ij}(x)}{ |\nabla \Phi_p|^{p-2}_A}  \frac{\partial |\nabla \Phi_p|_A^{p-2}}{\partial x_j}\frac{\partial r}{\partial x_i}.
\end{align*}
Now we apply equation \eqref{2.5} for $\partial r  / \partial x_i$
\begin{align*}
&=W |\nabla r|_A^p (v')^{p-2}\left( (p-1) v'' +  \left[\frac{Q-1}{r} +  \frac{W_r}{W}\right] v'\right) +  \left(\frac{p-1}{p-Q} \right)W |\nabla r|_A^{p-2} (v')^{p-1}\Phi_p^{\frac{p-1}{p-Q}-1}\\
& \times \frac{1}{|\nabla \Phi_p|^{p-2}_A}\underbrace{\sum_{i,j=1}^n\left[|\nabla \Phi_p|^{p-2}_Aa_{ij}(x)\frac{\partial^2 \Phi_p}{\partial x_i\partial x_j} + a_{ij}(x)\frac{\partial |\nabla \Phi_p|_A^{p-2}}{\partial x_j}\frac{\partial \Phi_p}{\partial x_i} + |\nabla \Phi_p|^{p-2}_A\frac{\partial a_{ij}}{\partial x_j}\frac{\partial \Phi_p}{\partial x_i}\right]}_{\mathcal{L}_{p,A}\Phi_p(x)=0} \\
& = W (r)|\nabla r|_A^p (v')^{p-2}\left( (p-1) v'' +  \left[\frac{Q-1}{r} +  \frac{W_r}{W}\right] v'\right).
\end{align*}
We arrive that \eqref{eq-PDE} can be rewritten as
\begin{equation}
	 W (r)|\nabla r|_A^p (v')^{p-2}\left( (p-1) v'' +  \left[\frac{Q-1}{r} +  \frac{W_r}{W}\right] v'\right) + |\nabla r|_A^pH(r) v^{p-1} =0,
\end{equation}
which means $(r^{Q-1} W(r) (v'(r))^{p-1})' + r^{Q-1}H(r) v^{p-1}(r) =0,$ which is \eqref{eq-ODE}.

The next theorem provides an explicit existence criterion of positive solution for ordinary differential equation \eqref{eq-ODE} which is proved by Agarwal-Bohner-Li \cite[Theorem 4.6.13]{ABL_book}:
\begin{thm}[Agarwal-Bohner-Li \cite{ABL_book}]\label{ABL-thm}
	Let $a:[r_0, \infty) \rightarrow (0,\infty)$ and $b:[r_0, \infty) \rightarrow (0,\infty)$ be  continuous functions with $b(r) \neq 0$. 
	Suppose that 
	\begin{equation*}
	\int_{r_0}^{\infty} b(s) ds< \infty, \,\, \text{and} \,\, \,\,	\phi(r) = 2 \int_{r}^{\infty} b(s)ds < \infty \,\,\, \text{for} \,\,\, r\geq r_0.
	\end{equation*}
	Suppose further that 
	\begin{equation}
	\int_{r_0}^{\infty} \left( \frac{\phi(s)}{a(s)}\right)^{\frac{1}{p-1}} ds  \leq \frac{1}{2(p-1)}.
	\end{equation}
	Then there exists a nonnegative solution to the following equation
	\begin{equation}
	(a(r) [y'(r)]^{p-1})' + b(r)[y(r)]^{p-1} =0 \,\,\, \text{for} \,\,\, r\geq r_0.
	\end{equation}
\end{thm}
The following theorem characterises the relation between $W(x)$ and $H(x)$ in order to obtain the weighted Hardy inequality:
\begin{thm}\label{main_thm}
	Let $\Omega$ be a bounded domain in $\mathbb{R}^n$. Let $W(x)$ and $H(x)$ be positive radially symmetric functions. Let $1<p<Q$. Let $d(x)$ be as in \eqref{quasi-norm}. Then the inequality 
	\begin{equation}
	\int_{\Omega} W(x) |\nabla u|^p_A dx \geq \int_{\Omega}|\nabla d|^p_A H(x) |u|^p dx
	\end{equation}
	holds for all complex-valued functions $u \in C^1_0(\Omega)$ provided that the following conditions hold: 
	\begin{equation}\label{2.9}
	\int_{r_0}^{\infty} s^{Q-1}H(s) ds< \infty, \,\, \text{and} \,\, \,\,	\phi(r) = 2 \int_{r}^{\infty} s^{Q-1}H(s) ds < \infty \,\,\, \text{for} \,\,\, r\geq r_0,
	\end{equation}
	\begin{equation}\label{2.10}
	\int_{r_0}^{\infty} \left( \frac{\phi(s)}{s^{Q-1}W(s)}\right)^{\frac{1}{p-1}} ds  \leq \frac{1}{2(p-1)} \,\,\, \text{for some} \,\,\, r_0>0.
	\end{equation}
\end{thm}

\begin{rem}	
	Note that
	\begin{itemize}
		\item As usual, we denote $W(x)=W(|x|)$ and $H(x)=H(|x|)$. We fix the notation for a positive function $f(x) > 0$ and a non-negative function $f(x)\geq 0$.
		\item For $p=2$, Theorem \ref{main_thm} answers to the question posed by Ghoussoub-Moradifam \cite{GM_book}.
		\item For $A(x)\equiv 1$, Theorem \ref{main_thm} was established for general (real-valued) vector fields with boundary terms by the authors \cite{RSS_NoDEA} (see also \cite{RS_book}, \cite{RV19} and \cite{Sabitbek_thesis}). 
	\end{itemize}
\end{rem}

To prove Theorem \ref{main_thm}, we need to establish the (first-order) Picone identity with $A(x)$ which is an $n\times n$ symmetric, uniformly positive definite matrix defined on $\Omega$. The proof of Lemma \ref{lem_Picone} is similar to the standard Picone identity obtained by Allegretto-Huang \cite{AH98} and the authors \cite{RSS_NoDEA}.
\begin{lem}\label{lem_Picone}
	Let $\Omega$ be a bounded domain in $\mathbb{R}^n$.
	Let a complex-valued function $u$ be differentiable a.e. in $\Omega$. Let $1<p<\infty$. Let a positive function $v$ be differentiable in $\Omega$. Define 
	\begin{align}
	R(u,v) &= |\nabla u|^p_A - \langle A(x)\nabla\left(\frac{|u|^p}{v^{p-1}} \right),|\nabla v|^{p-2}_A\nabla v\rangle, \\
	L(u,v) &= |\nabla u|^p_A -p\frac{|u|^{p-1} }{v^{p-1}}|\nabla v|^{p-2}_A \langle A(x)\nabla|u|, \nabla v \rangle + (p-1)\frac{|u|^p}{v^p}|\nabla v|^p_A,
	\end{align} 
	where $|\xi|^2_A= \langle A(x)\xi,\xi\rangle$. Then 
	\begin{equation*}
	L(u,v) =R(u,v)\geq 0.
	\end{equation*}
	Moreover, $L(u,v)=0$ a.e. in $\Omega$ if and only if $u\geq 0$ and $u=cv$ a.e. in $\Omega$ for some constant $c$ in each component of $\Omega$.	
\end{lem}
\begin{proof}[Proof of Lemma \ref{lem_Picone}]
	It is easy to show that $R(u,v)=L(u,v)$ by the expansion of $R(u,v)$ as follows
	\begin{align*}
	R(u,v) &= |\nabla u|^p_A - \langle A(x)\nabla\left(\frac{|u|^p}{v^{p-1}} \right),|\nabla v|^{p-2}_A\nabla v \rangle \\
	& = |\nabla u|^p_A -p\frac{|u|^{p-1} }{v^{p-1}}|\nabla v|^{p-2}_A \langle A(x) \nabla|u|, \nabla v\rangle + (p-1)\frac{|u|^p}{v^p}|\nabla v|^p_A\\
	& =  L(u,v).
	\end{align*}
	Let $u(x)=R(x)+iI(x)$, where $R(x)$ and $I(x)$ are the real and imaginary parts of $u$. We can restrict to the set where $u(x)\neq 0$. Then we have
	\begin{align}
	(\nabla |u|)(x) = \frac{1}{|u|} (R(x)\nabla R(x)+ I(x) \nabla I(x)).
	\end{align}
	Since
	\begin{align*}
	\left| \frac{1}{|u|} (R\nabla R+ I \nabla I) \right|^2_A \leq |\nabla R|^2_A + |\nabla I|^2_A,
	\end{align*}
	we get $|\nabla |u||_A\leq |\nabla u|_A$ a.e. in $\Omega$ (see \cite[Theorem 2.1]{RSS_revista}).

	Let us recall Young's inequality where for real numbers $a$ and $b$ we have
	\begin{equation*}
	p ab \leq a^p + (p-1)b^{\frac{p}{p-1}}.
	\end{equation*}
	By taking $a=|\nabla u|_A$ and $b=\frac{|u|^{p-1}}{v^{p-1}}|\nabla v|_A^{p-1}$, we prove $L(u,v)\geq 0$ in the following way:
	\begin{align*}
	L(u,v)& = |\nabla u|^p_A -p\frac{|u|^{p-1} }{v^{p-1}}|\nabla v|^{p-2}_A \langle A(x) \nabla|u|,\nabla v\rangle + (p-1)\frac{|u|^p}{v^p}|\nabla v|^p_A\\
	& =  |\nabla u|^p_A -p\frac{|u|^{p-1}}{v^{p-1}}|\nabla |u||_A|\nabla v|^{p-1}_A + (p-1)\frac{|u|^p}{v^p}|\nabla v|^p_A \\
	& + p\frac{|u|^{p-1}|\nabla v|^{p-2}_A}{v^{p-1}} \left( |\nabla |u||_A|\nabla v|_A -\langle A(x)\nabla |u|, \nabla v\rangle \right)  \\
	& \geq |\nabla u|^p_A -p\frac{|u|^{p-1}}{v^{p-1}}|\nabla u|_A|\nabla v|^{p-1}_A + (p-1)\frac{|u|^p}{v^p}|\nabla v|^p_A \\
	& + p\frac{|u|^{p-1}|\nabla v|^{p-2}_A}{v^{p-1}} \left( |\nabla |u||_A|\nabla v|_A - \langle A(x)\nabla |u|, \nabla v\rangle \right)\\
	& \geq p\frac{|u|^{p-1}|\nabla v|^{p-2}_A}{v^{p-1}} \left( |\nabla |u||_A|\nabla v|_A - \langle A(x)\nabla |u|, \nabla v\rangle \right).
	\end{align*}
	We will now show that $|\nabla |u||_A|\nabla v|_A \geq\langle A(x)\nabla |u|, \nabla v\rangle $, which implies $L(u,v)\geq 0$. A direct computation gives 
	\begin{align*}
	0 \leq | \nabla |u| - b\nabla v |^2_A &= \langle A(x) (\nabla |u| - b\nabla v), \nabla |u| - b\nabla v \rangle \\
	& =|\nabla |u||^2_A -2b\langle A(x) \nabla |u|,\nabla v\rangle + b^2 |\nabla v|^2_A.
	\end{align*}
	Setting $b=|\nabla v|_A^{-2} \langle A(x)\nabla |u|,\nabla v\rangle$ and rearranging produces 
	\begin{equation}\label{eq_2.11}
		|\nabla |u||_A|\nabla v|_A \geq\langle  A(x)\nabla |u|, \nabla v\rangle.
	\end{equation}
	Observe that $L(u,v)=0$ if and only if
	\begin{itemize}
		\item equality holds for $|\nabla |u||_A \leq |\nabla u|_A$ when $u\geq 0$;
		\item equality holds in \eqref{eq_2.11} when $u=cv$ for some constant $c$. 
	\end{itemize}
The proof is complete.
\end{proof}

\begin{proof}[Proof of Theorem \ref{main_thm}]
	By Theorem \ref{ABL-thm}, the conditions \eqref{2.9} and \eqref{2.10} provide the existence of a nonnegative solution to the following equation 
	\begin{equation}\label{eq-PDE1}
	\nabla \cdot (W(x) |\nabla v|^{p-2}_A A(x)\nabla v)+ |\nabla d|^p_A(x)H(x) v^{p-1} = 0.
	\end{equation}
	Then we prove by applying properties of the (first-order) Picone identity, divergence theorem and the equation \eqref{eq-PDE1}, respectively. We have  
	\begin{align*}
			0&\leq \int_{\Omega} W(x) R(u,v) dx\\
			&= \int_{\Omega} W(x)|\nabla u|^p_A dx - \int_{\Omega} W(x)  \langle A(x)\nabla\left(\frac{|u|^p}{v^{p-1}} \right), |\nabla v|^{p-2}_A\nabla v \rangle dx \\
			& = \int_{\Omega} W(x)|\nabla u|^p_A dx+ \int_{\Omega} \frac{|u|^p}{v^{p-1}} \nabla \cdot (W(x) |\nabla v|^{p-2}_A A(x)\nabla v)dx \\
			& =   \int_{\Omega} W(x)|\nabla u|^p_A - \int_{\Omega} |\nabla d|^p_AH(x) |u|^p dx.
	\end{align*}
This proves Theorem \ref{main_thm}.
\end{proof}

Next we will give examples for operators $\mathcal{L}_A$ by taking different matrices $A(x)$.

\bigskip

{\bf Euclidean Space $\mathbb{R}^n$:} Let $\Omega  = \mathbb{R}^n$. If we take $A(x)$ as an identity matrix, then $\mathcal{L}_A = - \Delta$ is the standard Laplacian, $\Phi(x)= |x|^{2-n}$ and $d(x) = |x|$ with $x \in \mathbb{R}^n$.
\begin{cor}
	Let $\Omega = \mathbb{R}^n$. Let $W(|x|)$ and $H(|x|)$ be positive  radially symmetric functions. Then the inequality 
	\begin{equation}
	\int_{\mathbb{R}^n} W(|x|) |\nabla u|^2 dx \geq \int_{\mathbb{R}^n} H(|x|) |u|^2 dx
	\end{equation}
	holds for all complex-valued functions $u \in C^1_0(\Omega)$ provided that the following conditions hold: 
	\begin{equation}
	\int_{r_0}^{\infty} s^{n-1}H(s) ds< \infty, \,\, \text{and} \,\, \,\,	\phi(r) = 2 \int_{r}^{\infty} s^{n-1}H(s) ds < \infty \,\,\, \text{for} \,\,\, r\geq r_0,
	\end{equation}
	\begin{equation}
	\int_{r_0}^{\infty}  \frac{\phi(s)}{s^{n-1}W(s)} ds  \leq \frac{1}{2} \,\,\, \text{for some} \,\,\, r_0>0.
	\end{equation}
\end{cor}
\medskip
{\bf The Heisenberg group $\mathbb{H}^1$:}  Let $\mathbb{H}^1:= \mathbb{R}^2 \times \mathbb{R}$ be Heisenberg group with $x = (x_1,x_2,x_3)$. We take  
$$A(x):=\begin{pmatrix}
1 & 0 & -\frac{x_2}{2}\\ 0 & 1 & \frac{x_1}{2}\\
-\frac{x_2}{2} & \frac{x_1}{2} & \frac{x_1^2 + x_2^2}{4}
\end{pmatrix}.$$
Then we have the following horizontal gradient 
$$\nabla_{\mathbb{H}}:= (\partial_{x_1} - \frac{x_2}{2} \partial_{x_3}, \partial_{x_2} + \frac{x_1}{2} \partial_{x_3}),$$ and the sub-Laplacian is given by
\begin{equation*}
\mathcal{L}_{\mathbb{H}} := \Delta_{x_1,x_2} + \frac{x_1^2 + x_2^2}{4} \partial_{x_3}^2 + (x_1\partial_{x_2} - x_2 \partial_{x_1})\partial_{x_3}.
\end{equation*}
The quasi-norm ($\mathcal{L}$-gauge) is given by 
\begin{equation*}
d_{\mathbb{H}}(x) = ( (x_1^2+x_2^2)^2 + 16 x_3^2 )^{\frac{1}{4}}. 
\end{equation*}
Note that the function $\Psi_{\mathcal{L}_{\mathbb{H}}}(x)$ could be explicitly calculated as follow
\begin{align*}
\Psi_{\mathcal{L}_{\mathbb{H}}}(x) = |\nabla_{\mathbb{H}} d_{\mathbb{H}}|^2(x) &= (X_1d_{\mathbb{H}})^2+ (X_2d_{\mathbb{H}})^2 \\
& =   d_{\mathbb{H}}^{-6} [(x_1^2+x_2^2)^2 x_1^2 - 8 (x_1^2+x_2^2)x_1x_2x_3 + 16x_1^2x_3^2] \\
& + d_{\mathbb{H}}^{-6} [(x_1^2+x_2^2)^2 x_2^2 + 8 (x_1^2+x_2^2)x_1x_2x_3 + 16x_2^2x_3^2] \\
& = d_{\mathbb{H}}^{-6}(x_1^2+x_2^2)[(x_1^2+x_2^2)^2 + 16x_3^2]\\
& = |x'|^2 d_{\mathbb{H}}^{-2}(x),
\end{align*}
where $|x'|^2 = x_1^2+ x_2^2$.
\begin{cor}\label{cor_H}
	Let $\Omega$ be a bounded domain in $\mathbb{H}^1$. Let $W(x)$ and $H(x)$ be positive radially symmetric functions. Then the inequality 
	\begin{equation}
	\int_{\Omega} W(x) |\nabla_{\mathbb{H}} u|^2 dx \geq \int_{\Omega} |\nabla_{\mathbb{H}} d_{\mathbb{H}}|^2 H(x) |u|^2 dx
	\end{equation}
	holds for all complex-valued functions $u \in C^1_0(\Omega)$ provided that the following conditions hold: 
	\begin{equation}
	\int_{r_0}^{\infty} s^{Q-1}H(s) ds< \infty, \,\, \text{and} \,\, \,\,	\phi(r) = 2 \int_{r}^{\infty} s^{Q-1}H(s) ds < \infty \,\,\, \text{for} \,\,\, r\geq r_0,
	\end{equation}
	\begin{equation}
	\int_{r_0}^{\infty}  \frac{\phi(s)}{s^{Q-1}W(s)} ds  \leq \frac{1}{2} \,\,\, \text{for some} \,\,\, r_0>0.
	\end{equation}
\end{cor}

\medskip	
{\bf Baouendi-Grushin operator:} 	Let $\Omega$ be an open subset of $\mathbb{R}^n=\mathbb{R}^k\times \mathbb{R}^l$ and $x\in \Omega$ with $x = (\xi,\zeta)$. For $\gamma>0$, we take  
	$$A(x):=\begin{pmatrix}
	I_k & 0\\ 0 &\gamma|\xi|^{\gamma}I_l
	\end{pmatrix},$$
	where $I_k$ and $I_l$ are the identity matrices of size $k$ and $l$, respectively. Then we have the following vector field $\nabla_{\gamma}:= (\nabla_{\xi}, \gamma|\xi|^{\gamma}\nabla_{\zeta})$ and the Baouendi-Grushin operator 
	$$\mathcal{L}_{\gamma}:= -\Delta_{\xi}-\gamma^2|\xi|^{2\gamma}\Delta_{\zeta}.$$
	For $x = (\xi,\zeta) \in \mathbb{R}^k\times \mathbb{R}^l$, let 
	\begin{equation*}
	d_{\gamma}(x) = (|\xi|^{2\gamma} + |\zeta|^2)^{1/2\gamma}.
	\end{equation*}
	
	As in the Heisenberg group, the function $\Psi_{\mathcal{L}_{\gamma}}(x)$ could be explicitly calculated as follow
\begin{align*}
\Psi_{\mathcal{L}_{\gamma}}(x) = |\nabla_{\gamma} d_{\gamma}|^2(x) &= \sum_{i=1}^k (\partial_{\xi_i} d_{\gamma})^2 + \sum_{i=1}^l \gamma^2|\xi|^{2\gamma}(\partial_{\zeta_i} d_{\gamma})^2 \\
& = (|\xi|^{2\gamma} + |\zeta|^2)^{\frac{1}{\gamma}-2} (|\xi|^{4\gamma} + |\xi|^{2\gamma}|\zeta|^2)\\
& = \frac{|\xi|^{2\gamma}}{d_{\gamma}^{2\gamma}(x)}.
\end{align*}	
\begin{cor}\label{cor_BG}
Let $\Omega$ be an open subset of $\mathbb{R}^n=\mathbb{R}^k\times \mathbb{R}^l$ and $x\in \Omega$ with $x = (\xi,\zeta)$. Let $W(x)$ and $H(x)$ be positive radially symmetric functions. Then the inequality 
	\begin{equation}
	\int_{\Omega} W(x) |\nabla_{\gamma} u|^2 dx \geq \int_{\Omega} |\nabla_{\gamma} d_{\gamma}|^{2} H(x) |u|^2 dx
	\end{equation}
	holds for all complex-valued functions $u \in C^1_0(\Omega)$ provided that the following conditions hold: 
	\begin{equation}
	\int_{r_0}^{\infty} s^{Q-1}H(s)ds< \infty, \,\, \text{and} \,\, \,\,	\phi(r) = 2 \int_{r}^{\infty} s^{Q-1}H(s) ds < \infty \,\,\, \text{for} \,\,\, r\geq r_0,
	\end{equation}
	\begin{equation}
	\int_{r_0}^{\infty} \frac{\phi(s)}{s^{Q-1}W(s)} ds  \leq \frac{1}{2}\,\,\, \text{for some} \,\,\, r_0>0.
	\end{equation}
\end{cor}	

\medskip
{\bf The Engel group $\mathbb{E}$:} Let $\mathbb{E}:= \mathbb{R}^2 \times \mathbb{R}\times \mathbb{R}$ be the Engel group with $x = (x_1,x_2,x_3,x_4)$. We take  
	$$A(x):=\begin{pmatrix}
	1 & 0 & -\frac{x_2}{2}& - \frac{x_3}{2} + \frac{x_1x_2}{12}\\ 0 & 1 & \frac{x_1}{2}& \frac{x_1^2}{12}\\
	-\frac{x_2}{2} & \frac{x_1}{2} & \frac{x_1^2+x_2^2}{4} &  \frac{x_2}{2}\left( \frac{x_3}{2} -\frac{x_1x_2}{12} \right) + \frac{x_1^3}{24} \\
	- \frac{x_3}{2} + \frac{x_1x_2}{12}&\frac{x_1^2}{12}& \frac{x_2}{2}\left( \frac{x_3}{2} -\frac{x_1x_2}{12} \right) + \frac{x_1^3}{24} & \left( \frac{x_3}{2} -\frac{x_1x_2}{12} \right)^2 + \frac{x_1^4}{144}
	\end{pmatrix}.$$
	Then the horizontal gradient and sub-Laplacian are given by 
	\begin{equation*}
	\nabla_{\mathbb{E}} : = (X_1,X_2),
	\end{equation*}  
	and
	\begin{align*}
	\mathcal{L}_{\mathbb{E}}:=   X_1^2 + X_2^2,
	\end{align*}
	where
	\begin{equation*}
	X_1:= \partial_{x_1} - \frac{x_2}{2} \partial_{x_3} - \left(\frac{x_3}{2}-\frac{x_1x_2}{12}\right)\partial_{x_4}, \,\, \text{and} \,\, X_2:=\partial_{x_2} + \frac{x_1}{2} \partial_{x_3} + \frac{x_1^2}{12} \partial_{x_4}.
	\end{equation*}

\begin{cor}
	Let $\Omega$ be a bounded domain in $\mathbb{E}$. Let $W(x)$ and $H(x)$ be positive radially symmetric functions. Then the inequality 
	\begin{equation}
	\int_{\Omega} W(x) |\nabla_{\mathbb{E}} u|^2 dx \geq \int_{\Omega} |\nabla_{\mathbb{E}} d|^2 H(x) |u|^2 dx
	\end{equation}
	holds for all complex-valued functions $u \in C^1_0(\Omega)$ provided that the following conditions hold: 
	\begin{equation}
	\int_{r_0}^{\infty} s^{Q-1}H(s) ds< \infty, \,\, \text{and} \,\, \,\,	\phi(r) = 2 \int_{r}^{\infty} s^{Q-1}H(s) ds < \infty \,\,\, \text{for} \,\,\, r\geq r_0,
	\end{equation}
	\begin{equation}
	\int_{r_0}^{\infty}  \frac{\phi(s)}{s^{Q-1}W(s)} ds  \leq \frac{1}{2} \,\,\, \text{for some} \,\,\, r_0>0.
	\end{equation}
\end{cor}

\medskip
{\bf The Cartan group $\mathcal{B}_5$:} Let $\mathcal{B}_5:= \mathbb{R}^2 \times \mathbb{R}\times \mathbb{R}^2$ be the Cartan group with $x = (x_1,x_2,x_3,x_4,x_5)$. We take  
	$$A(x):=\begin{pmatrix}
1 & 0 & 0& 0& 0\\ 
0 & 1 & -x_1& \frac{x_1^2}{2}& x_1x_2\\
0 & -x_1& x_1^2 & - \frac{x_1^3}{2} & -x_1^2x_2\\
0 & \frac{x_1^2}{2} & - \frac{x_1^3}{2} & \frac{x_1^4}{4} & \frac{x_1^3x_2}{2} \\
0 & x_1x_2 &-x_1^2x_2 &  \frac{x_1^3x_2}{2}& x_1^2x_2^2
\end{pmatrix}.$$

Then the horizontal gradient and sub-Laplacian are given by 
\begin{equation*}
\nabla_{\mathcal{B}_5} : = (X_1,X_2),
\end{equation*}  
and
\begin{align*}
\mathcal{L}_{\mathcal{B}_5}:=   X_1^2 + X_2^2,
\end{align*}
where
\begin{equation*}
X_1:= \partial_{x_1}  \,\, \text{and} \,\, X_2:=\partial_{x_2} - x_1  \partial_{x_3} + \frac{x_1^2}{2} \partial_{x_4} + x_1x_2\partial x_5.
\end{equation*}

\begin{cor}
	Let $\Omega$ be a bounded domain in $\mathcal{B}_5$. Let $W(x)$ and $H(x)$ be positive radially symmetric functions. Then the inequality 
	\begin{equation}
	\int_{\Omega} W(x) |\nabla_{\mathcal{B}_5} u|^2 dx \geq \int_{\Omega} |\nabla_{\mathcal{B}_5}d|^2H(x) |u|^2 dx
	\end{equation}
	holds for all complex-valued functions $u \in C^1_0(\Omega)$ provided that the following conditions hold: 
	\begin{equation}
	\int_{r_0}^{\infty} s^{Q-1}H(s) ds< \infty, \,\, \text{and} \,\, \,\,	\phi(r) = 2 \int_{r}^{\infty} s^{Q-1}H(s) ds < \infty \,\,\, \text{for} \,\,\, r\geq r_0,
	\end{equation}
	\begin{equation}
	\int_{r_0}^{\infty}  \frac{\phi(s)}{s^{Q-1}W(s)} ds  \leq \frac{1}{2} \,\,\, \text{for some} \,\,\, r_0>0.
	\end{equation}
\end{cor}

\section{Rellich inequality with Bessel pairs}\label{sec-Rel}
In this section, we present a Rellich inequlaity with Bessel pairs, which will be obtained as a byproduct of the (second-order) Picone type identity and the divergence theorem. 

\begin{thm}\label{thm_Rellich}
	Let $\Omega$ be a bounded domain in $\mathbb{R}^n$. Let $W \in C^{2}(\Omega)$ and $H \in L^1_{loc}(\Omega)$ be positive radially symmetric functions. Suppose that there exists a positive function $v\in C^{2}(\Omega)$ such that
	\begin{equation}\label{R-hyp_lp}
	\Delta (W(x) |\Delta v|^{p-2}\Delta v) \geq H(x) v^{p-1},
	\end{equation}
	with $-\Delta v >0$ a.e. in $\Omega$. Then for all complex-valued functions $ u \in C^{2}_0(\Omega)$, we have 
	\begin{equation}
	\int_{\Omega} W(x) |\Delta |u||^p dx \geq \int_{\Omega} H(x) |u|^p dx,
	\end{equation}
	where $1<p<n$.
\end{thm}
Here we present the corollary for $p=2$ to the above theorem:
\begin{cor}\label{cor_Rellich}
	Let $\Omega$ be a bounded domain in $\mathbb{R}^n$. Let $W \in C^{2}(\Omega)$ and $H \in L^1_{loc}(\Omega)$ be positive radially symmetric functions. Suppose that a positive function $v\in C^{\infty}(\Omega)$ satisfies
	\begin{equation}\label{R-hyp_l2}
	\Delta (W(x)\Delta v) \geq H(x) v,
	\end{equation}
	with $-\Delta v >0$ a.e. in $\Omega$. Then for all complex-valued functions $ u \in C^{2}_0(\Omega)$, we have 
	\begin{equation}\label{R_eq_l2}
	\int_{\Omega} W(x) |\Delta |u||^2 dx \geq \int_{\Omega} H(x) |u|^2 dx.
	\end{equation}
\end{cor}
In order to prove Theorem \ref{thm_Rellich}, we establish the (second-order) Picone type identity. 
\begin{lem}\label{lem_Piceone-R}
	Let $\Omega \subset \mathbb{R}^n$ be open set. Let $v$ be twice differentiable a.e. in $\Omega$ and satisfying the conditions $v >0$ and $-\Delta v >0$ a.e. in $\Omega$. Let a complex-valued function $u$ be twice differentiable a.e. in $\Omega$. For $p>1$ we define
	\begin{equation}
	R_1(u,v):= |\Delta |u||^p - \Delta\left( \frac{|u|^p}{v^{p-1}} \right) |\Delta v|^{p-2} \Delta v, 
	\end{equation}
	and 
	\begin{align}
	L_1(u,v):= &|\Delta |u||^p - p \left(\frac{|u|}{v}\right)^{p-1} \Delta |u| |\Delta v|^{p-2} \Delta v \\
	& + (p-1)\left( \frac{|u|}{v}\right)^p |\Delta v|^p  - p(p-1)\frac{|u|^{p-2}}{v^{p-1}} |\Delta v|^{p-2} \Delta v \left( \nabla |u| - \frac{|u|}{v} \nabla v\right)^2. \nonumber
	\end{align}
	Then we have 
	\begin{equation}
	L_1(u,v)=R_1(u,v) \geq 0.
	\end{equation}
\end{lem}
\begin{proof}[Proof of Lemma \ref{lem_Piceone-R}]
	We show that $R_1(u,v)=L_1(u,v)$ by a simple expansion of $R_1(u,v)$ as follows 
	\begin{align*}
	\Delta\left( \frac{|u|^p}{v^{p-1}} \right) &= \nabla \cdot \left( \frac{p|u|^{p-1}\nabla|u| }{v^{p-1}} - \frac{(p-1)|u|^p\nabla v}{v^p} \right)\\
	& = \sum_{i=1}^{n} \partial_{x_i} \left( \frac{p|u|^{p-1} \partial_{x_i}|u| }{v^{p-1}} - \frac{(p-1)|u|^p \partial_{x_i} v}{v^p} \right)\\
	& =  \sum_{i=1}^{n} \left[ \frac{ p(p-1)|u|^{p-2}( \partial_{x_i}|u|)^2 + p|u|^{p-1}\partial_{x_i}^2|u|  }{v^{p-1}} - \frac{p(p-1) |u|^{p-1} \partial_{x_i}|u|\partial_{x_i}v }{v^{p}} \right.\\
	& - \left.  \frac{p(p-1)|u|^{p-1}\partial_{x_i}|u|\partial_{x_i} v + (p-1)|u|^p \partial_{x_i}^2v }{v^p} +  \frac{p(p-1)|u|^p (\partial_{x_i} v)^2}{v^{p+1}} \right] \\
	& = p \frac{|u|^{p-1}}{v^{p-1}}\Delta |u| - (p-1)\frac{|u|^p}{v^p}\Delta v \\
	& + p(p-1)\left[ \frac{|u|^{p-2}}{v^{p-1}} |\nabla |u||^2 -2 \frac{|u|^{p-1}}{v^p}\langle \nabla |u|, \nabla v\rangle + \frac{|u|^{p}}{v^{p+1}} |\nabla v|^2 \right]\\ 
	& =  p \frac{|u|^{p-1}}{v^{p-1}}\Delta |u| - (p-1)\frac{|u|^p}{v^p}\Delta v  + p(p-1) \frac{|u|^{p-2}}{v^{p-1}} \left| \nabla |u| - \frac{|u|}{v} \nabla v \right|^2.
	\end{align*}
	The rest of proof is to apply Young's inequality, then we proceed as follows
	\begin{equation*}
	p	\frac{|u|^{p-1}}{v^{p-1}} \Delta |u| |\Delta v|^{p-2} \Delta v \leq |\Delta |u||^p + (p-1) \frac{|u|^p}{v^p}|\Delta v|^p,
	\end{equation*}
	where $p>1$. This gives
	\begin{align*}
	L_1(u,v)\geq - p(p-1)\frac{|u|^{p-2}}{v^{p-1}} |\Delta v|^{p-2} \Delta v \left( \nabla |u| - \frac{|u|}{v} \nabla v\right)^2.
	\end{align*}
	It is easy to see that $L_1(u,v)\geq 0$ by observing the fact $-\Delta v >0$.	
\end{proof}
\begin{proof}[Proof of Theorem \ref{thm_Rellich}]
	We prove by using the (second-order) Picone type identity and Green's second identity as follows:
	\begin{align*}
	0 &\leq \int_{\Omega} W(x) R_1(u,v) dx \\
	& = \int_{\Omega} W(x) |\Delta |u||^p dx - \int_{\Omega} W(x) \Delta \left(\frac{|u|^p}{v^{p-1}}\right) |\Delta v|^{p-2} \Delta v dx \\
	& = \int_{\Omega} W(x) |\Delta |u||^p dx -\int_{\Omega} \frac{|u|^p}{v^{p-1}} \Delta (W(x)|\Delta v|^{p-2} \Delta v  ) dx\\
	& = \int_{\Omega} W(x) |\Delta |u||^p dx -\int_{\Omega} H(x) |u|^p dx,
	\end{align*}
	using \eqref{R-hyp_lp}. This completes the proof.
\end{proof}
\subsection{Several versions of Rellich type inequalities}
Here by letting $W\equiv 1$ and $v=|x|^{-\frac{n-4}{2}}$ into \eqref{R-hyp_l2}, we obtain the function
\begin{equation*}
	H(x) = \frac{n^2(n-4)^2}{16}  |x|^{-4},
\end{equation*}
and inserting to inequality \eqref{R_eq_l2}, we have the following result:
\begin{cor}[Rellich inequality]
	Let $n\geq 5$. Then for all complex-valued functions $u  \in C_0^{\infty} (\mathbb{R}^n\backslash\{0\})$, we have 
	\begin{equation}
		\int_{\mathbb{R}^n} |\Delta |u||^2 dx \geq \frac{n^2(n-4)^2}{16} \int_{\mathbb{R}^n}  \frac{|u|^2}{|x|^4}dx.
	\end{equation}
\end{cor}

\begin{cor}\label{cor_Rel_p}
	Let $n\geq 3$ and $2 - \frac{n}{p}<\gamma<\frac{n(p-1)}{p}$. Then for all complex-valued functions $u  \in C_0^{\infty} (\mathbb{R}^n\backslash\{0\})$, we have 
	\begin{equation}\label{eq-Mid}
		\int_{\mathbb{R}^n} |x|^{\gamma p} |\Delta u|^p dx \geq \left( \frac{n}{p} -2 + \gamma \right)^p \left( \frac{n(p-1)}{p} -\gamma\right)^p \int_{\mathbb{R}^n}|x|^{(\gamma -2)p} |u|^p dx.
	\end{equation}
	In the case $\gamma=0$ and for $1<p<n/2$, we get 
		\begin{equation}\label{eq-Okaz}
	\int_{\mathbb{R}^n}  |\Delta u|^p dx \geq \left( \frac{n}{p} -2 \right)^p \left( \frac{n(p-1)}{p}\right)^p \int_{\mathbb{R}^n}|x|^{-2p} |u|^p dx.
	\end{equation} 
\end{cor}
\begin{rem}
	Note that the weighted Rellich inequality \eqref{eq-Mid} is proved by Mitidieri \cite{Mitidieri00} and $L^p$-Rellich inequality \eqref{eq-Okaz} by Okazawa \cite{Okazawa} with the optimal constants, respectively.
\end{rem}
\begin{proof}[Proof of Corollary \ref{cor_Rel_p}]
	Let us set 
	\begin{equation}
		W = |x|^{\gamma p} ,\,\, \text{and} \,\,\, v = |x|^{\alpha}, 
	\end{equation}
	where $\alpha = - (n/p +a -2)$. A direct computation gives 
	\begin{align*}
		\Delta v &= \alpha(\alpha + n-2) |x|^{\alpha-2}, \\
		|\Delta v|^{p-2} & = |\alpha|^{p-2}(\alpha + n-2)^{p-2} |x|^{(\alpha-2)(p-2)},\\
		W|\Delta v|^{p-2} \Delta v & = |\alpha|^{p-1} (\alpha + n-2)^{p-1}|x|^{(\alpha-2)(p-1) + \gamma p}.
	\end{align*}
	By inserting to \eqref{R-hyp_lp}, we arrive at 
	\begin{align*}
		\Delta (W|\Delta v|^{p-2} \Delta v) =  C_{\alpha,p,n,\gamma} |x|^{\alpha(p-1)  +(\gamma-2)p },
	\end{align*}
	where 
	$$C_{\alpha,p,n,\gamma}:=|\alpha|^{p-1} (\alpha + n-2)^{p-1}(\alpha p -\alpha -2p +2 + \gamma p) (\alpha p -\alpha -2p + \gamma p +n).$$
	Now we put the value of $\alpha$ in the constant, then we get 
	\begin{equation}
		H(x) = \left( \frac{n}{p} -2 + \gamma \right)^p \left( \frac{n(p-1)}{p} -\gamma\right)^p |x|^{(\gamma -2)p}.
	\end{equation}
	The statement then follows from Theorem \ref{thm_Rellich}.
\end{proof}


\begin{thebibliography}{NZW01}
\bibitem{ACR02}
Adimurthi, Chaudhuri, N., Ramaswamy, N.: An improved Hardy Sobolev inequality and its applications. Proc. Am. Math. Soc. 130, 489-505 (2002)

\bibitem{AS06}
Adimurthi, Sekar A.: Role of the fundamental solution in Hardy-Sobolev type inequalities. Proceedings of the Royal Society of Edinburgh 136A, 1111-1130 (2006)

	\bibitem{ABL_book}
Agarwal R. P., Bohner M., Li W-T.: Nonoscillation and oscillation: theory for functional differential equations, Dekker, New York, 1995
\bibitem{AH98}
Allegretto W., Huang Y.X.: A Picone's identity for the $p$-Laplacian and applications. Nonlinear Analysis, Theory, Methods and Applications 32(7), 819-830 (1998)	

\bibitem{Ancona}
	Ancona A.:
	\newblock On strong barriers and an inequalities of Hardy for domains $\mathbb{R}^n$. J. London Math. Soc. 43, 274-290 (1986) 

\bibitem{Avka_Lap}
Avkhadiev F.G. and Laptev A.:
\newblock On a sharp Hardy inequality for convex domains.
\newblock Springer; International Mathematical Series (New York) 12, Around the Research of Vladimir Maz'ya I, 1–12 (2010) 

\bibitem{Avk_Wirth}
Avkhadiev F. G. and Wirths K.J.:
\newblock Unified Poincar\'e and Hardy inequalities with sharp constants for convex domains.
\newblock ZAMM Z. Angew. Math. Mech. 87, 632-642 (2007)
	

\bibitem{BBDGV07}
Blanchet, A., Bonforte, M., Dolbeault, J., Grillo, G., Vasquez, J.L.: Hardy-Poincar\'e inequalities and applications to nonlinear diffusions. C. R. Acad. Sci. Paris, Ser. I 344, 431-436 (2007)

\bibitem{BG89}
Boccardo L., Galloutet Th.: Nonlinear elliptic and parabolic equations involving measure data. J. Funct. Anal. 87, 149-169 (1989) 
\bibitem{BLU07}
Bonfiglioli A., Lanconelli E. and Uguzzoni F.: Stratified Lie Groups and Potential Theory for their Sub-Laplacians. Springer-Verlag, Berlin-Heidelberg, 2007

\bibitem{BDE08}
Bosi R., Dolbeault J., and Esteban M.J.: Estimates for the optimal constants in multipolar Hardy inequalities for Schr\"odinger and Dirac operators. Commun. Pure Appl. Anal. 7, no. 3, 533-562 (2008)

\bibitem{BW09}
Bouchez V., Willem M.: Extremal functions for the Caffarelli-Kohn-Nirenberg inequalities: a simple proof of symmetry. J. Math. Anal. Appl. 352(1), 293-300 (2009)

\bibitem{BL85}
Brezis, H., Lieb, E.H.: Sobolev inequalities with remainder terms. J. Funct. Anal. 62, 73-86 (1985)

\bibitem{BM97}
Brezis, H., Marcus, M.: Hardy's inequality revisited. Ann. Scuola. Norm. Sup. Pisa 25, 217-237 (1997)

\bibitem{BMS00}
Brezis, H., Marcus, M., Shafrir, I.: Extremal functions for Hardy’s inequality with weight. J. Funct. Anal. 171, 177-191 (2000)

\bibitem{BV97}
Brezis, H., Vazquez, J.L.: Blow-up solutions of some nonlinear elliptic problems. Revista Mat. Univ. Complutense Madrid 10, 443-469 (1997)


\bibitem{CKN84}
Caffarelli, L., Kohn, R., Nirenberg, L.: First order interpolation inequalities with weights. Compos. Math. 53, 259-275 (1984)

\bibitem{CW01}
Catrina F., Wang Z.: On the Caffarelli-Kohn-Nirenberg inequalities: sharp constants, existence (and nonexistence), and symmetry of extremal functions. Comm. Pure Appl. Math. 54(2), 229-258 (2001)

\bibitem{Cazacu20}
Cazacu C.: The method of super-solutions in Hardy and Rellich inequalities in the $L^2$ setting: an overview of well-known results and short proofs.  Rev. Roumaine Math. Pures Appl. to appear. Preprint.

\bibitem{CZ13}
Cazacu C., Zuazua E.: Improved multipolar Hardy inequalities. Studies in phase space analysis with applications to PDEs, 35-52, Progr. Nonlinear Differential Equations Appl. 84, Birkhuser/Springer, New York, 2013

\bibitem{Cowan10}
Cowan C.: Optimal Hardy inequalities for general elliptic operators with improvements. Comm. Pure Appl. Anal. 9 no. 1, 109-140 (2010)

\bibitem{Davies99}
 Davies E. B.: A review of Hardy inequalities. The Mazya anniversary collection, Vol. 2 (Rostock, 1998), 55–67, Oper. Theory Adv. Appl., 110, Birkh\"auser, Basel, 1999

\bibitem{DAmbrasio04}
D'Ambrosio L.: Hardy-type inequalities related to degenerate elliptic differential operators. Ann. Scuola Norm. Sup. Pisa CI. Sci. 5, 451-486 (2005) 
	

\bibitem{FT06}
Felli V., Terracini S: Elliptic equations with multi-singular inverse-square potentials and critical nonlinearity. Comm. Partial Differential Equations 31, 469-495 (2006)

\bibitem{FR}
Fischer V., Ruzhansky M.:
\newblock Quantization on nilpotent Lie groups.
{Progress in Mathematics},  314, Birkh\"auser, (open access book) (2016)

\bibitem{GL90}
Garofalo N., Lanconelli E.: Frequency functions on the Heisenberg group, the uncertainty principle and unique continuation. Ann. Inst. Fourier (Grenoble), 40, No. 2, 313-356 (1990)

\bibitem{GM_book}
Ghoussoub N., Moradifam A.: Functional inequalities new perspectives and new applications. Mathematical Surveys and Monographs, 187, American Mathematical Society, Providence, RI 2013

\bibitem{GM11} 	
Ghoussoub N., Moradifam A.: Bessel pairs and optimal Hardy and Hardy-Rellich inequalities. Math. Ann. 349, 1-57 (2011) 

\bibitem{GKY17}
Goldstein J., Kombe I., and Yener A.: A unified approach to weighted Hardy type inequalities on Carnot groups. Discrete and Continuous Dynamical Systems 37, No. 4, 2009-2021 (2017)

\bibitem{Hardy}
Hardy G.: Note on a theorem of Hilbert. Math. Zeitschr. 6, 314-317 (1920)


\bibitem{KO90}
Kufner A. and Opic B.: “Hardy Type Inequalities,” Pitman Research Notes in Mathematics Series, 219. Longman Scientific and Technical, Harlow, 1990

\bibitem{KM92}
Kilpelainen T., Maly J.: Degenerate elliptic equations with measure data and nonlinear potentials. Ann. Scuola Norm. Sup. Pisa IV 19, 591-613, (1992)

\bibitem{Mazya85}
Maz’ya V.G.:  Sobolev Spaces, Berlin, Springer-Verlag, 1985

\bibitem{Mitidieri00}
Mitidieri E.: A simple approach to Hardy inequalities. Mathematical Notes 67, 479-486 (2000)

\bibitem{Okazawa} 
Okazawa N.: $L^p$-theory of Schr\"odinger operators with strongly singular potentials. Japan. J. Math. 22, 199--239 (1996)

\bibitem{RSS_revista}
Ruzhansky M., Sabitbek B., Suragan D.: Weighted $L^p$-Hardy and $L^p$-Rellich inequalities with boundary terms on stratified Lie groups. Rev. Mat. Complutense, 32, 19-35 (2019)
\bibitem{RSS_NoDEA}
Ruzhansky M., Sabitbek B., Suragan D.: Weighted anisotropic Hardy and Rellich type inequalities for general vector fields. NoDEA Nonlinear Differential Equations Appl. 26, no. 2, 26:13 (2019)

\bibitem{RS_book}
Ruzhansky M., Suragan D.:
\newblock Hardy inequalities on homogeneous groups.
{Progress in Math.} Vol. 327, Birkh\"auser, 588 pp, 2019

\bibitem{RV19}
Ruzhansky M., Verma D.: Hardy inequalities on metric measure spaces. Proc. R. Soc. A, 475, 20180310, 15pp. (2019)

\bibitem{RS_CCM_16}
Ruzhansky M., Suragan D.: Anisotropic $L^2$-weighted Hardy and $L^2$-Caffarelli-Kohn-Nirenberg inequalities.  Commun. Contemp. Math. 19(6), 1750014 (2017)

\bibitem{SS_18}
Sabitbek B., Suragan D.: Horizontal Weighted Hardy-Rellich Type inequality on Stratified Lie groups. Complex Anal. Oper. Theory. 12(6), 1469-1480 (2018)

\bibitem{Sabitbek_thesis}
Sabitbek B.: Hardy-Sobolev type inequalities on homogeneous groups and applications. PhD thesis, Al-Farabi Kazakh National University, (2019)

\bibitem{WW03}
Wang, Z.Q., Willem, M.: Caffarelli-Kohn-Nirenberg inequalities with remainder terms. J. Funct. Anal. 203, 550-568 (2003)


\end{thebibliography}
\end{document}